\documentclass[review,amsmath]{elsarticle}

\usepackage{amsmath,amsthm,amssymb,subeqnarray,amsfonts,graphics,color}
\usepackage{amsmath,amsthm,amssymb}
\usepackage{graphicx,subfigure}
\usepackage{tikz}
\usepackage{epstopdf}

\usepackage{setspace}
\usepackage{stmaryrd}
\usepackage{cases}
\usepackage{exscale}
\usepackage{relsize}

\usepackage{graphicx}
\usepackage{float}

\usepackage{booktabs}
\usepackage{threeparttable}
\usepackage{multirow}

\usepackage{algorithm}
\usepackage{algpseudocode}
\usepackage{amsmath}
\allowdisplaybreaks[4]

\usepackage[footnotesize]{caption2}

\numberwithin{equation}{section}
\numberwithin{table}{section}
\numberwithin{figure}{section}

\newtheorem{theorem}{Theorem}[section]

\newtheorem{remark}{Remark}[section]
\newtheorem{lemma}{Lemma}[section]

\biboptions{numbers,sort&compress}

\journal{}
\begin{document}

\begin{frontmatter}
\title{The compactness of Moser-Trudinger functionals with conical metric in the unit ball of $\mathbb{R}^N$}
\tnotetext[t1]{ This work is partially supported by the NSF of China under grant Nos. 12001466 and  U19A2079.}

\author[a,b]{Qi Xia \corref{cor1}}
\ead{xq19991231@mail.dlut.edu.cn}

\address[a]{School of Mathematical Sciences,
Dalian University of Technology, Dalian, 116024, China}

\address[b]{DUT-BSU Joint Institute,
Dalian University of Technology, Dalian, 116024, China}

\cortext[cor1]{Corresponding author}

\begin{abstract}
Let $\mathbb{B}$ be the unit ball in $\mathbb{R}^N$, $W_0^{1,N} \left( \mathbb{B} \right)$ is a standard Sobolev space. Zhang proved the extremal function of the Moser-Trudinger inequality as follows,  
\begin{flalign*}
	\int_{ \mathbb{B}} h_{ \epsilon} (x) e^{ \alpha_N \left( 1 + \epsilon \right) |u_{\epsilon}|^{ \frac{N}{N-1} } } dx, ~ u_{\epsilon} \in W_0^{1,N} \left(  \mathbb{B} \right) \cap \mathcal{S},
\end{flalign*}
where $\alpha_N = \omega_N^{ \frac{1}{N-1} }$, $\omega_N $ is the area of the unit sphere in $\mathbb{R}^N$(see \citep{26}) .
In this paper, we consider the compactness of the sequence $\{ u_{\epsilon} \}_{\epsilon} $ and prove that it has a subsequence converging to a function in $C^1 \left(\overline{ \mathbb{B}} \right)$. 
\end{abstract}

\begin{keyword}
Moser-Trudinger Inequality, conical metric, blow-up analysis, compactness
\end{keyword}

\end{frontmatter}

\section{Introduction}
Let $\Omega$ be a smooth bounded domains in $\mathbb{R}^N$, and $ W_0^{1,N} \left( \Omega \right)$ is a standard Sobolev space with norm
\begin{flalign*}
	|| \cdot ||_{W_0^{1,N} \left( \Omega \right) } := \left( \int_{\Omega} |\nabla \cdot|^N dx \right)^{ \frac{1}{N} },
\end{flalign*}
where $\nabla$ is the gradient operator. Then the classical Moser-Trudinger is
\begin{flalign*}
	\underset{u \in W_0^{1,N} \left( \Omega \right) \setminus \{0\}, ~ ||u||_{ W_0^{1,N} \left( \Omega \right) } \leq 1 }{\sup} \int_{\Omega} e^{ \alpha |u|^{ \frac{N}{N-1} } } dx < \infty.
\end{flalign*} 
For more details and  the origin of the Moser-Trudinger inequality, see \citep{20,21}. This inequality is sharp in the sense that, for every $\alpha > \alpha_N$, all integrals above are finite, but the supremum is not, for instance, \cite{17}, while the existence of extremal functions was established in \citep{4,5}.

As $\epsilon \rightarrow 0^+$, the following result holds from \citep{6},
\begin{flalign}
	\underset{u \in W_0^{1,2} \left( \mathbb{B} \right) \bigcap \mathcal{S} \setminus \{0 \}, ||u||_{ W_0^{1,2} \left( \mathbb{B} \right) } \leq 1 }{\sup} \int_{\mathbb{B}} |x|^{ 2 \epsilon } e^{4 \pi \left( 1 +\epsilon \right) u^2 } dx < \infty, \label{eq: 1.1}
\end{flalign}
where $\mathcal{S} $ denotes the set of all radially symmetric functions. Also, Calanchi, Terraneo, et al. \cite{7,8,16}  proved the existence of extremal functions $ u_{\epsilon} \in W_0^{1,2} \left( \mathbb{B} \right) \bigcap \mathcal{S}$, namely,
\begin{flalign}
	\underset{u \in W_0^{1,2} \left( \mathbb{B} \right) \bigcap \mathcal{S} \setminus \{0 \}, ||u||_{ W_0^{1,2} \left( \mathbb{B} \right) } \leq 1 }{\sup} \int_{\mathbb{B}} |x|^{ 2 \epsilon } e^{4 \pi \left( 1 + \epsilon \right) u^2 } dx = \int_{\mathbb{B}} |x|^{ 2 \epsilon } e^{4 \pi \left( 1 + \epsilon \right) u_{\epsilon}^2 } dx. \label{eq: 1.2}
\end{flalign}

Yang and Zhu extended equation \eqref{eq: 1.1} to high dimensions in \citep{1}.
Equation \eqref{eq: 1.2} also holds in high dimensions, as shown in \citep{2}. In particular,  the following embedding is continuous,  $W_0^{1,N} \left( \mathbb{B}\right) \cap \mathcal{S} \hookrightarrow L^p \left( \mathbb{B}, ~ |x|^{\alpha} \right)$ for $ \alpha >0$ and $p = 2 \frac{N + \alpha}{N-2}$. In \citep{26}, Zhang proved that there exists a function that attains the supremum of the functional,
\begin{flalign}
	\int_{\mathbb{B}} h_{\epsilon} (x) e^{ 4 \pi ( 1+ \epsilon) u^2 } dx, \label{eq: 1.5}
\end{flalign}
where $h_{\epsilon}(x)$ satisfies $\underset{\epsilon \rightarrow 0^+}{\lim} h_{\epsilon}(x) |x|^{ -2 \epsilon} =1$. 

Recently, Yang and Wang considered the compactness of extremals for critical singular Trudinger-Moser functionals in \cite{18} in the case $\epsilon < 0$. Meanwhile, this work also provided a more delicate method than that in \cite{19} to analyze the behavior of asymptotic symmetry. In \cite{3}, Shan and Li considered the Moser-Trudinger functional with the conical metric in the unit ball of $\mathbb{R}^2$. They also obtained that for any $\epsilon >0$, there exist radially symmetric functions $u_{\epsilon}$ satisfying
\begin{flalign*}
	\int_{\mathbb{B} } |x|^{N \epsilon} e^{ \alpha_N \left( 1+ \epsilon \right) |u_{\epsilon}|^{ \frac{N}{N-1} } } dx = \underset{ u \in W_0^{1,N} \left( \mathbb{B} \right) \cap \mathcal{S} \setminus \{0\}, ~ ||u||_{ W_0^{1,N} \left( \mathbb{B} \right) } \leq 1 }{\sup} \int_{\mathbb{B} } |x|^{N \epsilon} e^{ \alpha_N \left( 1+ \epsilon \right) |u|^{ \frac{N}{N-1} } } dx.
\end{flalign*}
We then consider whether the results hold for the generalized functions $h_{\epsilon} (x)$. The principal contributions of this work are outlined below,
\begin{theorem}
	Assume that $u_{\epsilon}$ is a sequence of maximizers for the supremum mentioned in \eqref{eq: 1.5}. Suppose a function $h_{\epsilon}(x)$ is radially symmetric, nonnegative,
	continuous on $\overline{\mathbb{B}}$ and satisﬁes $h_{\epsilon} (x) > 0$ on $\overline{ \mathbb{B} } \setminus \{0\}$, satisfying the following,
	\begin{flalign*}
		& \underset{\epsilon \rightarrow 0}{\lim} h_{\epsilon} (x) = |x|^{N \epsilon} \cdot h_0 (x),\\
		& \underset{x \rightarrow 0}{\lim} h_{\epsilon} (x) |x|^{ - N \epsilon} =1.
	\end{flalign*} 
	Then, up to a subsequence, there exists a function $u_0$ such that $u_{\epsilon } \rightarrow u_0$ in $C^1 ( \overline{\mathbb{B}} )$ and $u_0$ is an extremal function attaining the supremum
	\begin{flalign}
		\underset{u \in W_0^{1, N} ( \overline{\mathbb{B}} ), ~ ||u||_{W_0^{1, N} ( \overline{\mathbb{B}}) }  \leq  1}{\sup} \int_{\mathbb{B} } h_{0} (x) e^{ \alpha_N |u|^{ \frac{N}{N-1} } }  dx. \label{1.3}
	\end{flalign}
\end{theorem}

The outline of the proof of Theorem 1.1 is given below. Assume that $u_{\epsilon} $ are as mentioned in equation \eqref{1.3}. They are also the solutions of Euler--Lagrange equations below. If $u_{\epsilon}$ are not bounded, then there exists a constant $A_0$ satisfying
\begin{flalign*}
	\underset{u \in W_0^{1,N} \left( \mathbb{B} \right) \setminus \{0 \}, ||u||_{ W_0^{1,N} \left( \mathbb{B} \right) } \leq 1 }{\sup} \int_{\mathbb{B}} h_{ \epsilon } (x) e^{\alpha_N |u|^{ \frac{N}{N-1} } } dx  \leq \int_{ \mathbb{B} } h_0(x) dx + e^{ \alpha_N A_0 + \sum_{j=1 }^{N-1} \frac{1}{j} }.
\end{flalign*} 
On the other hand, we also find that there exist functions $u$ such that
\begin{flalign*}
	\underset{u \in W_0^{1,N} \left( \mathbb{B} \right) \setminus \{0 \}, ||u||_{ W_0^{1,N} \left( \mathbb{B} \right) } \leq 1 }{\sup} \int_{\mathbb{B}} h_{\epsilon } (x) e^{\alpha_N |u|^{ \frac{N}{N-1} } } dx  > \int_{\mathbb{B}} h_0(x) + e^{ \alpha_N A_0 + \sum_{j=1 }^{N-1} \frac{1}{j} }.
\end{flalign*} 
The above two inequalities yield a contradiction, which in turn implies the uniform boundedness of $u_{\epsilon}$ on $\mathbb{B}$. Then,  applying elliptic estimates, we obtain the desired result immediately.

\section{Some Lemmas}

Theorem 1.1 is proved by means of a blow-up analysis analogous to the one developed in \cite{13,14, 22, 23}, and we will divide the proof into several steps.

\subsection{The Euler--Lagrange equation of $u_{\beta}$}

From equation \eqref{1.3}, there exist some $u_{\epsilon} \in W_0^{1,N}\left(\mathbb{B} \right)$ satisfying $|| \nabla u_{\epsilon}||_{L^N (\mathbb{B})} = 1$ and 
\begin{flalign*}
	\mathcal{J_{\epsilon}} = \int_{\mathbb{B} } h_{ \epsilon} (x) e^{ \alpha_N (1 + \epsilon) |u_{\epsilon}|^{\frac{N}{N-1} } }  dx.
\end{flalign*}
It is not difficult to find that $u_{\epsilon}$ satisfy the Euler--Lagrange equations
\begin{equation}
	\begin{cases}
		& - \operatorname{div} (|\nabla u_{\epsilon}|^{N-2} \cdot  \nabla u_{\epsilon}) = \frac{1}{\lambda_{\epsilon}} h_{ \epsilon}(x) |u_{\epsilon}|^{ \frac{1}{N-1} }  e^{ \alpha_N (1 + \epsilon) |u_{\epsilon}|^{\frac{N}{N-1} } }, ~ \text{in} ~ \mathbb{B};\\
		& u_{\epsilon} \geq 0, ||u_{\epsilon}||_{W_0^{1,N} (\mathbb{B})} =1;\\
		& \lambda_{\epsilon} = \int_{\mathbb{B}} h_{ \epsilon} (x) |u_{\epsilon}|^{ \frac{N}{N-1} }  e^{ \alpha_N (1 + \epsilon) |u_{\epsilon}|^{\frac{N}{N-1} }  }dx. \label{eq: 2.1}
	\end{cases}	
\end{equation}
Since $u_{\epsilon}$ is bounded in $W_0^{1,N} (\mathbb{B})$, there exists  a subsequence (still denoted by $\{u_{\epsilon}\}_{\epsilon}$)
\begin{equation}
	\begin{cases}
		& u_{\epsilon} \rightharpoonup u_0, ~ W_0^{1,N} (\mathbb{B});\\
		& u_{\epsilon} \rightarrow u_0, ~ L^p (\mathbb{B});\\
		&  u_{\epsilon} \rightarrow u_0, ~ \text{a.e.} ~ \mathbb{B},
	\end{cases} \label{eq: 2.2}
\end{equation}
 as $\epsilon \rightarrow 0$. By the Lebesgue dominated convergence theorem and Fotou's lemma, we have
\begin{flalign*}
	& \int_{\mathbb{B}} h_{0} (x) e^{ \alpha_N |u|^{\frac{N}{N-1} }  } dx \\
	& = \underset{\epsilon \rightarrow 0}{\lim} \int_{\mathbb{B}} h_{ \epsilon} (x) e^{ \alpha_N (1 + \epsilon) |u|^{\frac{N}{N-1} }  } dx\\
	& \leq \underset{ \epsilon \rightarrow 0}{\lim \inf} \mathcal{J}_{\epsilon} \\
	& = \underset{ \epsilon \rightarrow 0}{\lim \inf} \int_{\mathbb{B}} h_{ \epsilon} (x) e^{ \alpha_N (1 + \epsilon) |u_{\epsilon}|^{\frac{N}{N-1} }  } dx. 
\end{flalign*}
Then we obtain 
\begin{flalign*}
	\mathcal{J} \leq \underset{ \epsilon \rightarrow 0}{\lim \inf} \int_{\mathbb{B}} h_{ \epsilon} (x) e^{ \alpha_N (1 + \epsilon) |u_{\epsilon}|^{\frac{N}{N-1} }  } dx. 
\end{flalign*}
Next, we prove that $\lambda_{\epsilon} > 0$ has a lower bound. Since $t^{ \frac{N}{N-1} } e^{t^{\frac{N}{N-1} }} \geq e^{t^{ \frac{N}{N-1} }  } - 1$ for $t \geq 0$, there holds 
\begin{flalign*}
	& \lambda_{\epsilon} = \int_{\mathbb{B}} h_{\epsilon} (x) |u_{\epsilon}|^{ \frac{N}{N-1} }  e^{ \alpha_N (1 + \epsilon) |u_{\epsilon}|^{\frac{N}{N-1} }  }dx\\
	& \geq \frac{1}{\alpha_N (1+ \epsilon)} \int_{\mathbb{B}} \alpha_N (1+ \epsilon) h_{ \epsilon} (x) |u_{\epsilon}|^{ \frac{N}{N-1} }  e^{ \alpha_N (1 + \epsilon) |u_{\epsilon}|^{ \frac{N}{N-1} }  }dx\\
	& \geq  \frac{1}{\alpha_N (1+ \epsilon)} \int_{\mathbb{B}} h_{ \epsilon} (x) \left( e^{ \alpha_N (1 + \epsilon) |u_{\epsilon}|^{ \frac{N}{N-1} }  } - 1 \right)dx.
\end{flalign*}
From then on, as $\epsilon \rightarrow 0$, we get that 
\begin{flalign*}
	\underset{\epsilon \rightarrow 0}{\lim \inf} \lambda_{\epsilon} \geq \frac{1}{\alpha_N} \left(\mathcal{J} - \int_{ \mathbb{B} } h_0 (x) dx \right) > 0; 
\end{flalign*}

We let $c_{\epsilon} = u_{\epsilon} (0) = \underset{\mathbb{B}}{\max}u_{\epsilon}$. There are two possibilities for the analysis of $u_{\epsilon}$, either $u_{\epsilon}$ are bounded in $\mathbb{B}$, or $\underset{\mathbb{B}}{\max}u_{\epsilon} \rightarrow + \infty$ as $\epsilon \rightarrow 0^+$. Then we describe the blow-up phenomenon of $u_{\epsilon}$.

\subsection{Blow-up Analysis}

 \begin{lemma}
 	The function $u_0 \equiv 0$ and $|\nabla u_{\epsilon}|^N dx \rightharpoonup \delta_0$ in the sense of measure, where $\delta_0$ denotes the usual Dirac measure giving  unit mass to the point 0.
 \end{lemma}

\begin{proof}
	We denote that
	\begin{flalign*}
		f_{\epsilon} (x) =\frac{1}{\lambda_{\epsilon}} h_{ \epsilon} (x) e^{ \alpha_N (1 + \epsilon) |u_{\epsilon}|^{\frac{N}{N-1} } }.
	\end{flalign*}
    We can find that $\int_{\mathbb{B}} |f_{\epsilon}|^p dx \leq C < + \infty $ by Lion's concentration-compactness principle in \citep{24,25}.
    Applying elliptic estimates to equations, we obtain that $u_{\epsilon}$ is uniformly bounded on $\mathbb{B}$. This contradicts that $c_{\epsilon} \rightarrow \infty$ as $\epsilon \rightarrow 0^+$. Therefore $u_0 \equiv 0$.
        
    Next, We need to prove that $|\nabla u_{\epsilon}| \rightharpoonup \delta_0$ as $\epsilon \rightarrow 0^+$.  Firstly, we have that $||u_{\epsilon}||_{W^{1,N} (\mathbb{B})} =1$. If the assertion is flase, then there would exist some $0 < r_0 <1$ such that 
        \begin{flalign*}
        	\underset{\epsilon \rightarrow 0}{\lim \sup} \int_{\mathbb{B}_{r_0}} |\nabla u_{\epsilon} |^N dx \leq \gamma < 1.  
        \end{flalign*}
        It is not difficult to see that $\operatorname{div} (|\nabla u_{\epsilon}|^{N-2} \cdot \nabla u_{\epsilon} )$ is bounded in $L^q (\mathbb{B}_{\frac{r_0}{2}})$ for some $q>1$. From elliptic estimates for the equation, we have that $u_{\epsilon}$ is bounded in $\mathbb{B}_{\frac{r_0}{4}}$, again contradicting $c_{\epsilon} \rightarrow + \infty$. We then obtain $|\nabla u_{\epsilon}| \rightharpoonup \delta_0$ as $\epsilon \rightarrow 0$. Thus the lemma is proved.
\end{proof}

Let $r_{\epsilon} = \lambda_{\epsilon}^{\frac{1}{N}} c_{\epsilon}^{- \frac{1}{N-1} } e^{ -\frac{\alpha_N (1+ \epsilon)}{N} c_{\epsilon}^{ \frac{N}{N-1} } }$ and $\tilde{u}_{\epsilon} = \tilde{u}_{\epsilon} (x) : = u_{\epsilon} (r_{\epsilon}^{ \frac{1}{1+ \epsilon} } x)$, from \citep{1} and the classical Trudinger-Moser inequality, we know that 
\begin{flalign*}
	r_{\epsilon} e^{\alpha |u_{\epsilon}|^{ \frac{N}{N-1} } } \rightarrow 0
\end{flalign*} 
as $\epsilon \rightarrow 0$, $\forall \alpha < \frac{\alpha_N(1+\epsilon)}{N}$. For $x \in \mathbb{B}_{r_{\epsilon}^{-1}  }$, we define two sequences of functions by
\begin{flalign*}
	\psi_{\epsilon} (x) = c_{\epsilon}^{-1} u_{\epsilon} (r_{\epsilon}^{ \frac{1}{1+ \epsilon} } x) ,
\end{flalign*}
and 
\begin{flalign*}
	\phi_{\epsilon} (x) = c_{\epsilon}^{ \frac{1}{N-1}  }\left( u_{\epsilon} (r_{\epsilon}^{ \frac{1}{1+ \epsilon} } x) -c_{\epsilon} \right)
\end{flalign*}
for $x \in \mathbb{B}_{\epsilon} = \{ x \in \mathbb{R}^N: ~ r_{\epsilon}^{\frac{1}{1+\epsilon }} x \in \mathbb{B} \}$.
Then, we consider the asymptotic behavior of $u_{\epsilon}$ at the origin.

\begin{lemma}
	Up to a subsequence, there holds $\psi_{\epsilon} \rightarrow 1$ in $C^1_{loc} (\mathbb{R}^N)$ and $\phi_{\epsilon} \rightarrow \phi$ in $C^1_{loc} (\mathbb{R}^N)$ as $\epsilon \rightarrow 0$, where 
	\begin{flalign*}
		\phi(x) = - \frac{N-1}{\alpha_N} \ln \left(1+ \frac{\alpha_N}{N^{ \frac{N}{N-1} }} |x|^{\frac{N}{N-1}  } \right). 
	\end{flalign*}
\end{lemma}

\begin{proof}
	Direct calculation shows
    \begin{flalign*}
    	& - \operatorname{div} \left(  |\nabla \psi_{\epsilon} (x) |^{N-2} \cdot \nabla \psi_{\epsilon} (x) \right) \\
    	& = \left(c_{\epsilon}^{-1}  \right)^{N-1} r_{\epsilon}^{N } \cdot \frac{1}{\lambda_{\epsilon}} h_{ \epsilon} (x)\psi_{\epsilon}^{ \frac{1}{N-1} }  e^{ \alpha_N (1 + \epsilon) |\tilde{u}_{\epsilon}|^{\frac{N}{N-1} } }\\
    	& = c_{\epsilon}^{ - N } h_{ \epsilon} (x) \psi_{\epsilon}^{ \frac{1}{N-1} } \exp \left\lbrace  \alpha_N (1 + \epsilon) \left(|\tilde{u}_{\epsilon}|^{\frac{N}{N-1} }- c_{\epsilon}^{ \frac{N}{N-1} } \right) \right\rbrace 
    \end{flalign*}
     We also have
    \begin{flalign*}
    	& - \operatorname{div} \left(  |\nabla \phi_{\epsilon} (x) |^{N-2} \cdot \nabla \phi_{\epsilon} (x) \right) \\
    	& = h_{\epsilon} (x) \psi_{\epsilon}^{ \frac{1}{N-1} } \exp \left\lbrace  \alpha_N (1 + \epsilon) \left(|\tilde{u}_{\epsilon}|^{\frac{N}{N-1} }- c_{\epsilon}^{ \frac{N}{N-1} } \right) \right\rbrace .
    \end{flalign*}
     We note that $| \psi_{\epsilon} |\leq 1$, $r_{\epsilon} \rightarrow 0$, $c_{\epsilon} \rightarrow + \infty$ as $\epsilon \rightarrow 0$. Then we have $\operatorname{div} \left( |\nabla \psi_{\epsilon}|^{N-2} \cdot \nabla \psi_{\epsilon} \right) \in L^{ \frac{N}{N-1} } (\mathbb{B}_{\epsilon} )$. From elliptic estimates in \cite{9}, we get
     \begin{flalign*}
     	\psi_{\epsilon} \rightarrow \psi_0, ~ \text{in} ~ C^1_{loc} (\mathbb{R}^N) ~ \text{as} ~ \epsilon \rightarrow 0,
     \end{flalign*}
     where
     \begin{flalign*}
     	\psi_0 (0) = 1, ~ 0 \leq \psi_0 (x) \leq 1, ~ \forall x \in \mathbb{R}^N.
     \end{flalign*}
     Also, The function $\psi_0$ satisfies $- \operatorname{div} \left( |\nabla \psi_{0}|^{N-2} \cdot \nabla \psi_{0} \right) =0 $ in $\mathbb{R}^N$. From the Liouville theorem, we deduce that  $\psi_0 (0) \equiv \psi_0 (x) = 1$.
     
     Also we have by applying elliptic estimates in \cite{4} that,
     \begin{flalign*}
     	\phi_{\epsilon} \rightarrow \phi,~ \text{in} ~ C^1_{loc} (\mathbb{R}^N) ~ \text{as} ~ \epsilon \rightarrow 0,
     \end{flalign*}
     where $\phi$ satisﬁes
      \begin{equation*}
     	\begin{cases}
     		& - \operatorname{div} \left( |\nabla \phi|^{N-2} \cdot \nabla \phi \right) = |x|^{N\epsilon} \exp \{ \alpha_N \frac{N}{N-1} \phi \}, ~ \text{in} ~ \mathbb{R}^N;\\
     		& \underset{\mathbb{R}^N}{\sup} \phi = \phi(0) = 0.
     	\end{cases}
     \end{equation*}
     Thus the results hold.
\end{proof}

\begin{remark} \label{rem: 2.1}
    This lemma means that the weak function $\phi$ satisfies
    \begin{equation*}
    	\begin{cases}
    		& - \operatorname{div} \left( |\nabla \phi|^{N-2} \cdot \nabla \phi \right) = |x|^{N \epsilon} \exp \{ \alpha_N \frac{N}{N-1} \phi \}, ~ \text{in} ~ \mathbb{R}^N;\\
    		& \underset{\mathbb{R}^N}{\sup} \phi = \phi(0) = 0.
    	\end{cases}
    \end{equation*}
     Moreover, $\phi(x) = - \frac{N-1}{\alpha_N \left(1 + \epsilon \right) } \ln \left(1+ \left( \frac{N \left(1 + \epsilon \right) }{\omega_N} \right)^{ \frac{1}{N-1} } |x|^{\frac{N \left( 1+ \epsilon \right) }{N-1}  } \right)$ and, 
     \begin{flalign}
     	\int_{\mathbb{R}^N} |x|^{N \epsilon} e^{ \alpha_N \frac{N}{N-1} \phi } dx = 1. \label{eq: 2.3}
     \end{flalign}
\end{remark}

We proceed to analyze the convergence of $u_{\epsilon} $ away from zero. As in \cite{12}, we set $u_{\epsilon, \iota} = \min \{ \iota c_{\epsilon}, ~ u_{\epsilon} \}$ for any $\iota \in (0,1) $.
 
\begin{lemma}
	For any $ 0 < \iota <1$, we have 
	\begin{equation*}
		\underset{\epsilon \rightarrow 0}{\lim} \int_{\mathbb{B}} |\nabla u_{\epsilon, \iota} |^N dx = \iota.
	\end{equation*}
\end{lemma}

\begin{proof}
	By the divergence theorem, we have
     \begin{flalign*}
     	\int_{\mathbb{B}} |\nabla u_{\epsilon, \iota} |^N dx & = -\int_{\mathbb{B}} u_{\epsilon, \iota} \operatorname{div} \left( | \nabla u_{\epsilon}|^{N-2} \cdot \nabla u_{\epsilon} \right) dx\\
     	& \geq \frac{1}{\lambda_{\epsilon}} \int_{\mathbb{B}_{ R_{\epsilon}^{ \frac{1}{ 1+ \epsilon } } } (0) }  h_{ \epsilon} (x) u_{\epsilon}^{ \frac{1}{N-1} } \cdot u_{\epsilon,\iota}  \exp \{ \alpha_N (1 + \epsilon) |u_{\epsilon}|^{\frac{N}{N-1} } \} dx\\
     	& = \iota ( 1 + o_{\epsilon} (1)) \cdot\\
        & \int_{\mathbb{B}_R (0)} |y|^{N \epsilon} \exp \left\lbrace  \alpha_N ( 1+ \epsilon ) \left(|\tilde{u}_{\epsilon}|^{ \frac{N}{N-1} } -  c_{\epsilon}^{ \frac{N}{N-1} } \right) \right\rbrace  dy\\
        & + o_{\epsilon}(1).
     \end{flalign*}
     Testing the equation \eqref{eq: 2.1} by $(u_{\epsilon} - \iota c_{\epsilon})^+$, and we get that 
     \begin{flalign*}
     	\int_{\mathbb{B}} |\nabla (u_{\epsilon} - \iota c_{\epsilon})^+ |^N dx & = -\int_{\mathbb{B}} (u_{\epsilon} - \iota c_{\epsilon})^+ \operatorname{div} \left( | \nabla (u_{\epsilon} - \iota c_{\epsilon})^+|^{N-2} \cdot \nabla (u_{\epsilon} - \iota c_{\epsilon})^+ \right) dx\\
     	& \geq \frac{1}{\lambda_{\epsilon}} \int_{\mathbb{B}_{ R_{\epsilon}^{ \frac{1}{ 1+ \epsilon } } } (0) }  h_{ \epsilon} (x) u_{\epsilon} \cdot ( u_{\epsilon} - \iota c_{\epsilon} )^+  \exp \{ \alpha_N (1 + \epsilon) |u_{\epsilon}|^{\frac{N}{N-1} } \} dx\\
     	& = (1 - \iota) ( 1 + o_{\epsilon} (1)) \cdot\\
     	& \int_{\mathbb{B}_R (0)} |y|^{2 \epsilon} \exp \left\lbrace  \alpha_N ( 1+ \epsilon ) \left(|\tilde{u}_{\epsilon}|^{ \frac{N}{N-1} } - c_{\epsilon}^{ \frac{N}{N-1} } \right) \right\rbrace  dy\\
     	& + o_{\epsilon}(1).
     \end{flalign*}
     In fact, let $t_{\epsilon}$ lie between $u_{\epsilon}(r_{\epsilon}^{ \frac{1}{1+ \epsilon} } x )$ and $c_{\epsilon}$. By the mean value theorem, we have that
     \begin{flalign*}
     	& |\tilde{u}_{\epsilon}|^{ \frac{N}{N-1} } - c_{\epsilon}^{ \frac{N}{N-1} }\\
      	& \geq \frac{t_{\epsilon}^{ \frac{1}{N-1} }}{ c_{\epsilon}^{ \frac{1}{N-1} } }  c_{\epsilon}^{ \frac{1}{N-1} } ( | \tilde{u}_{\epsilon}| - c_{\epsilon} )\\
      	& = \frac{t_{\epsilon}^{ \frac{1}{N-1} }}{ c_{\epsilon}^{ \frac{1}{N-1} } } \phi_{\epsilon} .
      \end{flalign*}
     From Fatou's Lemma and equation (2.3), we have
     \begin{flalign}
     	& \underset{\epsilon \rightarrow 0}{\lim \inf} \int_{\mathbb{B}} |\nabla u_{\epsilon, \iota}|^N dx \geq \iota ; \label{eq: 2.4} \\
     	& \underset{\epsilon \rightarrow 0}{\lim \inf} \int_{\mathbb{B}} |\nabla (u_{\epsilon} - \iota c_{\epsilon} )|^N dx \geq 1 - \iota . \label{eq: 2.5}
     \end{flalign}
     We also note that 
     \begin{flalign}
     	\int_{\mathbb{B}} |\nabla u_{\epsilon, \iota}|^N dx + \int_{\mathbb{B}} | \nabla (u_{\epsilon} - \iota c_{\epsilon})^+ |^N dx = \int_{\mathbb{B}} |\nabla u_{\epsilon}|^N dx = 1+ o_{\epsilon} (1). \label{eq: 2.6}
     \end{flalign} 
     Combining \eqref{eq: 2.4}, \eqref{eq: 2.5} and \eqref{eq: 2.6}, we have that the results hold.
\end{proof}

\begin{lemma}
	There holds 
	\begin{flalign*}
		\underset{\epsilon \rightarrow 0}{\lim \sup}\int_{\mathbb{B}} h_{ \epsilon} (x)\exp \{ \alpha_N (1 + \epsilon) u_{\epsilon}^{ \frac{N}{N-1} } \} dx & \leq \int_{\mathbb{B}} h_0(x) dx \\
		& + \underset{R \rightarrow + \infty}{\lim} \underset{\epsilon \rightarrow 0}{\lim \sup} \int_{ \mathbb{B}_{R_{ r_{\epsilon}^{ \frac{1}{1+\epsilon} } }} } h_{ \epsilon} (x) \exp \{ \alpha_N (1 + \epsilon) u_{\epsilon}^{ \frac{N}{N-1} } \} dx.
	\end{flalign*}
\end{lemma}

\begin{proof}
	For any $0<\iota<1$, we have
	\begin{flalign*}
		\int_{\mathbb{B}} h_{\epsilon} (x) \exp \{ \alpha_N ( 1+ \epsilon) u_{\epsilon}^{ \frac{N}{N-1} }  \} & = \int_{ u_{\epsilon} \leq \iota c_{\epsilon} } h_{\epsilon} (x) \exp \{ \alpha_N ( 1+ \epsilon) u_{\epsilon}^{ \frac{N}{N-1} }  \} dx \\
		&+ \int_{ u_{\epsilon} > \iota c_{\epsilon} } h_{\epsilon} (x) \exp \{ \alpha_N ( 1+ \epsilon) u_{\epsilon}^{ \frac{N}{N-1} }  \} dx.
	\end{flalign*}
    The first lemma implies that $u_{\epsilon,\iota}$ converges to $0$ a.e. in $\mathbb{B}$, from which we obtain
    \begin{flalign*}
    	\exp \{ \alpha_N ( 1+ \epsilon) u_{\epsilon}^{ \frac{N}{N-1} }  \} \rightarrow \exp \{ 0 \}
    \end{flalign*}
     a.e. in $\mathbb{B}$.
     \begin{flalign*}
     	\int_{ u_{\epsilon} \leq \iota c_{\epsilon} } h_{\epsilon} (x) \exp \{ \alpha_N ( 1+ \epsilon) u_{\epsilon}^{ \frac{N}{N-1} }  \} dx & \leq \int_{ \mathbb{B} } h_{\epsilon} (x) \exp \{ \alpha_N ( 1+ \epsilon) u_{\epsilon, \iota}^{ \frac{N}{N-1} }  \} dx\\
     	& = \int_{\mathbb{B} } h_0(x) dx + o_{ \epsilon } (1).
     \end{flalign*}
     In additional, we estimate
     \begin{flalign*}
     	\int_{ u_{\epsilon} > \iota c_{\epsilon} } h_{\epsilon} (x) \exp \{ \alpha_N ( 1+ \epsilon) u_{\epsilon}^{ \frac{N}{N-1} }  \} dx & \leq \frac{1}{\iota^{ \frac{N}{N-1} } c_{\epsilon}^{ \frac{N}{N-1} } } \int_{ u_{\epsilon} > \iota c_{\epsilon} } h_{ \epsilon} (x) |u_{\epsilon}|^{ \frac{N}{N-1} } \exp \{ \alpha_N ( 1+ \epsilon) u_{\epsilon}^{ \frac{N}{N-1} }  \} dx\\
     	& \leq \frac{1}{\iota^{ \frac{N}{N-1} } c_{\epsilon}^{ \frac{N}{N-1} } } \int_{ \mathbb{B} } h_{\epsilon} (x) |u_{\epsilon}|^{ \frac{N}{N-1} } \exp \{ \alpha_N ( 1+ \epsilon) u_{\epsilon}^{ \frac{N}{N-1} }  \} dx\\
     	& = \frac{ \lambda{\epsilon} }{ \iota^{ \frac{N}{N-1} } c_{\epsilon}^{ \frac{N}{N-1} } }.
     \end{flalign*}
      Letting $\iota \rightarrow 1$, then we have
      \begin{flalign}
      	\underset{ \epsilon \rightarrow 0}{ \lim \sup } \int_{\mathbb{B}} h_{\epsilon} (x) \exp \{ \alpha_N (1+ \epsilon) u_{\epsilon}^{ \frac{N}{N-1} } \} dx \leq \int_{\mathbb{B}} h_0 (x) dx + \underset{\epsilon \rightarrow 0}{\lim \sup} \frac{\lambda_{\epsilon}}{  c_{\epsilon}^{ \frac{N}{N-1} } }. \label{eq: 2.7}
      \end{flalign}
      On the other hand,  for any fixed $R > 0$, we obtain
      \begin{flalign*}
      	\int_{\mathbb{B}_{R r_{\epsilon}^{ \frac{1}{1+ \epsilon} } } } h_{\epsilon } (x) \exp \{ \alpha_N (1+ \epsilon) u_{\epsilon}^{ \frac{N}{N-1} } \} dx & = \int_{\mathbb{B}_R } |y|^{ N \epsilon } r_{\epsilon}^N \exp \{ \alpha_N (1+ \epsilon) \tilde{u}_{\epsilon}^{ \frac{N}{N-1} } \} dy + o_{\epsilon} (1)\\
      	& = \lambda_{\epsilon} c_{\epsilon}^{- \frac{N}{N-1} } \cdot \int_{\mathbb{B}_R} \exp \{ \alpha_N \frac{N}{N-1} \phi \} dx+o_{\epsilon}(1).
      \end{flalign*}
      Thus,
      \begin{flalign}
      	& \underset{R \rightarrow + \infty}{\lim } \underset{\epsilon \rightarrow 0}{\lim \sup} \int_{\mathbb{B}_{R r_{\epsilon}^{ \frac{1}{1+ \epsilon} } } } |x|^{ N \epsilon } \exp \{ \alpha_N (1+ \epsilon) u_{\epsilon}^{ \frac{N}{N-1} } \} dx  \notag\\
      	& = \underset{R \rightarrow + \infty}{\lim } \underset{\epsilon \rightarrow 0}{\lim \sup} \lambda_{\epsilon} c_{\epsilon}^{- \frac{N}{N-1}  } \cdot \int_{\mathbb{B}_R} \exp \{ \alpha_N \frac{N}{N-1} \phi \} dx \notag \\
      	& \geq \underset{\epsilon \rightarrow 0}{\lim \sup} \frac{\lambda_{\epsilon}}{  c_{\epsilon}^{ \frac{N}{N-1} } }. \label{eq: 2.8}
      \end{flalign}
      The lemma follows from equations \eqref{eq: 2.7} and \eqref{eq: 2.8}.
\end{proof}

\begin{lemma}
	For any $\alpha <  \frac{N}{N-1} $, we have that
	\begin{flalign*}
		\underset{\epsilon \rightarrow 0}{\lim}  \frac{\lambda_{\epsilon}}{c_{\epsilon}^{\alpha}} = + \infty.
	\end{flalign*}
\end{lemma}

\begin{proof}
	If not, then, $\frac{\lambda_{\epsilon}}{c_{\epsilon}^{ \frac{N}{N-1} }} \rightarrow 0$ as $\epsilon \rightarrow 0$. For any $u \in W_0^{1,N} ( \mathbb{B})$ with $|| u||_{ W_0^{1,N} (\mathbb{B}) } = 1$. According to the equations \eqref{eq: 2.1} and \eqref{eq: 2.7}, we have
	\begin{flalign*}
		\int_{\mathbb{B}} h_0 (x) dx \leq \int_{\mathbb{B}} h_0 (x) \exp \{ \alpha_N \frac{N}{N-1} |u|^{ \frac{N}{N-1} }  \} dx \leq \mathcal{J } & \leq \underset{\epsilon \rightarrow 0}{\lim \inf} \int_{\mathbb{B}} h_{\epsilon} (x) \exp \{ \alpha_N (1 + \epsilon) |u_{\epsilon}|^{\frac{N}{N-1} }  \} dx\\
		& \leq \int_{\mathbb{B}} h_0 (x) dx + \underset{\epsilon \rightarrow 0}{\lim \sup} \frac{\lambda_{\epsilon}}{c_{\epsilon}^{ \frac{N}{N-1} }} =\int_{\mathbb{B}} h_0 (x) dx.
	\end{flalign*}
    It is impossible since the left hand integral is stractly greater than $|\mathbb{B}|$. This result is proved.
\end{proof}

We shall now discuss the convergence of $c_{\epsilon}^{ \frac{1}{N-1} } u_{\epsilon}$ under the assumption $c_{\epsilon} \rightarrow + \infty$.

\begin{lemma} \label{lem: 2.6}
	The sequence satisfies the following relations
	\begin{flalign}
		& c_{\epsilon}^{\frac{1}{N-1}} u_{\epsilon} \rightharpoonup G_0, ~ W_0^{1, q} (\mathbb{B}), ~ 1< q < N, \notag \\
		&c_{\epsilon}^{\frac{1}{N-1}} u_{\epsilon} \rightarrow G_0, ~ L^p (\mathbb{B}), ~ 1 < p< \frac{Nq}{N-q}, \notag \\
		&c_{\epsilon }^{ \frac{1}{N-1}} u_{\epsilon} \rightarrow G_0,~ C_{loc}^1 (\mathbb{B} \setminus \{0\}). \label{eq: 2.9}
	\end{flalign}
    Here $G_0$ is a Green's function and satisfies $- \operatorname{div} (| \nabla G_0|^{N-2} \cdot \nabla G_0) = \delta_0$ in the distributional sense, where $\delta_0$ is the usual Dirac measure centered at $0$.
\end{lemma}

\begin{proof}
	Let 
	\begin{flalign*}
		g_{\epsilon} (x) = \lambda_{\epsilon}^{-1} h_{\epsilon}(x) c_{\epsilon}^{ \frac{1}{N-1} } u_{\epsilon} \exp \{ \alpha_{\epsilon} ( 1 + \epsilon) |u_{\epsilon}|^{ \frac{N}{N-1} } \}.
	\end{flalign*}
    For any $\varphi \in C_0^{\infty} ( \mathbb{B} )$, we have 
    \begin{flalign*}
    	\underset{ \epsilon \rightarrow 0}{\lim } \int_{ \mathbb{B}} g_{\epsilon} (x) \varphi(x) dx = \varphi(0). 
    \end{flalign*}
    We divide the integral into two parts, i.e,
    \begin{flalign*}
    	\int_{\mathbb{B}} g_{\epsilon} (x) \varphi (x) dx = \int_{u_{\epsilon } \leq \iota c_{\epsilon} } g_{\epsilon} (x) \varphi (x) dx + \int_{u_{\epsilon } > \iota c_{\epsilon}} g_{\epsilon} (x) \varphi (x) dx.
    \end{flalign*}
    We estimate the two integrals on the above equation. For the first part, we obtain
    \begin{flalign*}
    	\int_{ u_{\epsilon} \leq \iota c_{\epsilon} } g_{\epsilon} (x) \varphi (x) dx & = \int_{ u_{\epsilon} \leq \iota c_{\epsilon} } \lambda_{\epsilon}^{-1} h_{\epsilon} (x) c_{\epsilon}^{ \frac{1}{N-1} } u_{\epsilon} \exp \{ \alpha_{\epsilon} ( 1 + \epsilon) |u_{\epsilon}|^{ \frac{N}{N-1} } \} \varphi(x) dx\\
    	& \leq \frac{ c_{\epsilon}^{ \frac{1}{N-1} } }{ \lambda_{\epsilon} } \left(\underset{\mathbb{B}}{\sup} \varphi(x) \right) \int_{ \mathbb{B} } h_{\epsilon} (x) u_{\epsilon} \exp \{ \alpha_{\epsilon} ( 1 + \epsilon) |u_{\epsilon}|^{ \frac{N}{N-1} } \} dx.
    \end{flalign*}
    Note that $u_{\epsilon} \rightarrow u_0$ strongly in $L^p (\mathbb{B})$ for $p>1$, then by H\"{o}lder inequality, we have
    \begin{flalign*}
    	& \int_{ \mathbb{B} } h_{\epsilon} (x) u_{\epsilon} \exp \{ \alpha_{\epsilon} ( 1 + \epsilon) |u_{\epsilon}|^{ \frac{N}{N-1} } \} dx\\
    	& \leq \left( \int_{ \mathbb{B} } h_{\epsilon} (x)  \exp \{ \alpha_{\epsilon} p ( 1 + \epsilon) |u_{\epsilon}|^{ \frac{N}{N-1} } \} dx \right)^{ \frac{1}{p} } \left( \int_{ \mathbb{B} } h_{\epsilon} (x) |u_{\epsilon}|^q   dx \right)^{ \frac{1}{q} } \leq C,
    \end{flalign*}
    where $\frac{1}{p} + \frac{1}{q} =1$. According to the lemma above, we have
    \begin{flalign*}
    	\int_{ u_{\epsilon} \leq \iota c_{\epsilon} } g_{\epsilon} (x) \varphi (x) dx = o_{\epsilon} (1).
    \end{flalign*}
    By the lemma \ref{lem: 2.6}, it follows that $\mathbb{B}_{R_{ \epsilon }^{ \frac{1}{1+ \epsilon} } } \subset \{ u_{\epsilon} > \iota \epsilon \}$ for $\epsilon$ small enough. Then we can implies 
    \begin{flalign*}
    	\int_{ \mathbb{B}_{R_{ \epsilon }^{ \frac{1}{1+ \epsilon} } } \cap \{ u_{\epsilon} > \iota \epsilon \} } g_{\epsilon} (x) \varphi (x) dx & = \int_{ \mathbb{B}_{R_{ \epsilon }^{ \frac{1}{1+ \epsilon} } } } g_{\epsilon} (x) \varphi (x) dx\\
    	& = \int_{ \mathbb{B}_{R_{ \epsilon }^{ \frac{1}{1+ \epsilon} } } } \lambda_{\epsilon}^{-1} h_{\epsilon} (x) c_{\epsilon}^{ \frac{1}{N-1} } u_{\epsilon} \exp \{ \alpha_{\epsilon} ( 1 + \epsilon) |u_{\epsilon}|^{ \frac{N}{N-1} } \} \varphi (x) dx\\
    	& = c_{\epsilon}^{- \frac{1}{N-1}  } \cdot\\
    	& \int_{\mathbb{B}_R } |y|^{N \epsilon } \tilde{u}_{\epsilon} \exp \{ \alpha_N (1+ \epsilon) |\tilde{u}_{\epsilon}|^{ \frac{N}{N-1} } - \alpha_N ( 1+ \epsilon) c_{\epsilon}^{ \frac{N}{N-1} }  \}  \tilde{\varphi}  dy\\
    	& + o_{\epsilon}(1),
    \end{flalign*}
    where $\tilde{\varphi} = \varphi( r_{\epsilon}^{ \frac{1}{1 + \epsilon} } x )$. From the remark \ref{rem: 2.1}, we have that
    \begin{flalign*}
    	\int_{ \mathbb{B}_{R_{ \epsilon }^{ \frac{1}{1+ \epsilon} } } \cap \{ u_{\epsilon} > \iota \epsilon \} } g_{\epsilon} (x) \varphi (x) dx & = \varphi(0) \left( 1 + o_{\epsilon} (1) \right) \left( 1 + o_{R} (1) \right) \\
    	& =  \varphi(0) \left( 1 + o_{\epsilon} (1) + o_{R} (1) \right).
    \end{flalign*}
    On the other hand, we obtain
    \begin{flalign*}
    	\int_{\{ u_{\epsilon} > \iota c_{\epsilon} \}\setminus \mathbb{B}_{ R_{ r_{\epsilon}^{ \frac{1}{1 + \epsilon} } } }  }  g_{\epsilon} (x) \varphi(x) dx &  \leq \frac{1}{\iota}
    	\left( \underset{\mathbb{B}}{\sup} |\varphi(x)| \right) \lambda_{\epsilon}^{-1} \cdot\\
    	& \int_{\{ u_{\epsilon} > \iota c_{\epsilon} \}\setminus \mathbb{B}_{ R_{ r_{\epsilon}^{ \frac{1}{1 + \epsilon} } } }  } |x|^{N \epsilon} c_{\epsilon}^{ \frac{1}{N-1} } u_{\epsilon} \exp \{ \alpha_{\epsilon} ( 1 + \epsilon) |u_{\epsilon}|^{ \frac{N}{N-1} } \} dx \\
    	& \leq \frac{1}{\iota}
    	\left( \underset{\mathbb{B}}{\sup} |\varphi(x)| \right) \left( 1- \int_{\mathbb{B}_R} \exp \{ \alpha_N \frac{N}{N-1} \phi \}dx + o_{\epsilon} (1)    \right)\\
    	& = o_{\epsilon} (1) + o_R (1).
    \end{flalign*}  
    We have
    \begin{flalign*}
    	\underset{\epsilon \rightarrow 0}{\lim} \int_{u_{\epsilon} > \iota c_{\epsilon}} g_{\epsilon} (x) \varphi (x) dx = \varphi(0),
    \end{flalign*}
    by $R \rightarrow + \infty$.
    Then $\varphi(x)$ is the Dirac function. Multiplying both sides of the equation \eqref{eq: 2.1} by $c_{\epsilon}^{ \frac{1}{N-1} }$, we have 
    \begin{flalign*}
    	- \operatorname{div} (|\nabla c_{\epsilon}^{ \frac{1}{N-1} } u_{\epsilon}|^{N-2} \cdot  \nabla c_{\epsilon}^{ \frac{1}{N-1} } u_{\epsilon}) = \frac{1}{\lambda_{\epsilon}} h_{\epsilon} (x) c_{\epsilon}^{ \frac{1}{N-1} } u_{\epsilon}  \exp \{ \alpha_N (1 + \epsilon) |u_{\epsilon}|^{\frac{N}{N-1} } \}, ~ \text{in} ~ \mathbb{B}.
    \end{flalign*}
    Then we conclude that $h_{\epsilon} (x)$ is bounded in $L^1_{loc}(\mathbb{B})$. 
    Applying the results above, similar methods of proof in \cite{10}, and the Sobolev embedding
    theorem, we obtain a subsequence $\{ u_{\epsilon} \}_{\epsilon}$ satisfying 
    \begin{equation*}
    	\begin{cases}
    		& c_{\epsilon}
    		^{
    			\frac{1}{N-1}
    		} u_{\epsilon} \rightharpoonup G_0, ~ \text{in}
    		~ W_0^{1,q}(\mathbb{B}), ~ \forall 1 < q < N;		\\
    		&c_{\epsilon}
    		^{
    			\frac{1}{N-1}
    		} u_{\epsilon} \rightarrow G_0, ~ \text{in}
    		~ L^{s}(\mathbb{B}), ~ \forall 1 \leq s \leq \frac{Nq}{N-q};	\\
    		&c_{\epsilon}
    		^{
    			\frac{1}{N-1}
    		} u_{\epsilon} \rightarrow G_0, ~ \text{in}
    		~ C^1_{loc} \left(\overline{\mathbb{B}} \setminus \{0\} \right),
    	\end{cases}
    \end{equation*}
    where $G_0$ is a weak solution of 
    \begin{equation*}
    	- \operatorname{div} \left( |\nabla G_0|^{N-2} \cdot \nabla G_0 \right) = \delta_0 ~ \text{in} ~ \mathbb{B}.
    \end{equation*}

    We proceed to analyze the convergence of $u_{\epsilon} $ away from zero. Moreover,  $G_0$ takes the form 
    \begin{flalign*}
    	G_0 = - \frac{N}{ \alpha_N } \ln |x| + A_0 + \beta (x)
    \end{flalign*}
     by \cite{11}, where $A_0$ is a constant, $\beta (x) \in C^{1, \gamma} (\mathbb(\overline{B}) )$ for $0 < \gamma <1$, $\gamma(0) = 0 $.
\end{proof}
    
\begin{lemma}
     Let $\zeta_{\epsilon} (x) \in W_0^{1,N} ( \mathbb{B} )$ with $\int_{\mathbb{B} } |\zeta_{\epsilon} (x)|^N dx \leq 1$, $ \zeta_{\epsilon} (x) \rightharpoonup 0$ weakly in $W_0^{1,N} ( \mathbb{B} )$ as $\epsilon \rightarrow 0$. $ \zeta_{\epsilon} (x) $ is nonnegative and radially symmetric. Then we have 
     \begin{flalign*}
     	\underset{\epsilon \rightarrow 0}{ \lim \sup } \int_{\mathbb{B}} h_{\epsilon} (x) \left( \exp \{ \alpha_N (1 + \epsilon) |\zeta_{\epsilon}|^{ \frac{N}{N-1} } \} -1 \right) dx \leq \frac{\omega_N}{N} \exp \left\lbrace \alpha_N A_0 + \sum_{j=1}^{N-1} \frac{1}{j} \right\rbrace  .
     \end{flalign*}
\end{lemma}
 
\begin{proof}
     By the radial symmetry of $\zeta_{\epsilon} (x)$, we have $ \zeta_{\epsilon} (x) = \zeta_{\epsilon} (r)$ with $r=|x|$, and we make the change of variables,
     \begin{flalign*}
     	\rho_{\epsilon} (r) = (1 + \epsilon)^{ \frac{N-1}{N} } \zeta_{\epsilon} \left( r^{ \frac{1}{ 1 + \epsilon} } \right).
     \end{flalign*}
    A straightforward calculation shows that
     \begin{flalign*}
     	\int_{\mathbb{B}} | \nabla \rho_{\epsilon} (x)|^N dx & = \int_{[0, \pi]^{N-2} \times [0, 2 \pi]} d \omega(\theta, \phi) \int_0^1 | \nabla (1 + \epsilon)^{ \frac{N-1}{N} } \zeta_{\epsilon} \left( s \right)|^N r^{N-1}dr\\
     	& = \omega_N \int_0^1 | \zeta_{\epsilon}^{ \prime } (r) |^N r^{N-1} dr\\
     	& = \int_{\mathbb{B}} | \nabla \zeta_{\epsilon} \left( x \right)|^N dx.
     \end{flalign*}
     Then above estimate implies that $|| \rho_{\epsilon} ||_{ W_0^{1, N}  \left( \mathbb{B} \right) }=   || \zeta_{\epsilon} ||_{ W_0^{1, N}  \left( \mathbb{B} \right) } \leq 1$. Passing to a subsequence, we may assume that 
     \begin{equation*}
     	\begin{cases}
     		& \rho_{\epsilon} \rightharpoonup \rho_* , ~ W_0^{1, N} \left( \mathbb{B} \right);\\
     		& \rho_{\epsilon} \rightarrow \rho_* , ~ L^p \left( \mathbb{B}  \right), ~ \forall p > 0;\\
     		& \rho_{\epsilon} \rightarrow \rho_* , ~ \text{a.e.} ~ \mathbb{B}.
     	\end{cases}
     \end{equation*}
     It is obvious to see $\zeta_{\epsilon} \rightarrow 0$ a.e. in $\mathbb{B}$, and then it follows that $ \rho_* \rightarrow 0$ a.e. in $\mathbb{B}$. Furthermore, by a change of variable $r = \mu^{ \frac{1}{1 + \epsilon} }$, it holds that
     \begin{flalign*}
     	\int_{ \mathbb{B}} h_{\epsilon} (x) \left\lbrace \exp \{ \alpha_N \left( 1 + \epsilon \right) \zeta^{ \frac{N}{N-1} }_{ \epsilon } (x)  \} - 1 \right\rbrace dx  & = \omega_N \int_0^1 h_{\epsilon} (r)|r|^{ N-1 } \left\lbrace \exp \{ \alpha_N \left( 1 + \epsilon \right) \zeta^{ \frac{N}{N-1} }_{ \epsilon }
        (r)  \} - 1 \right\rbrace dr\\
         & \leq \frac{1}{1 + \epsilon} \int_{\mathbb{B}} h_{\epsilon} (x)\left\lbrace \exp \{ \alpha_N \rho_{\epsilon}^{ \frac{N}{N-1} } (x) -1  \} \right\rbrace dx\\
         & \leq \frac{1}{1+ \epsilon} \frac{\omega_N}{N} \exp \left\lbrace \alpha_N A_0 + \sum_{j=1}^{N-1} \frac{1}{j} \right\rbrace ,
     \end{flalign*}
     where $A_0$ is contained in the Green function. It follows from that a result of Yang's reference \cite{10} that
     \begin{flalign*}
     	\underset{\epsilon \rightarrow 0}{\lim \sup} \int_{ \mathbb{B}} h_{\epsilon} (x) \exp \left\lbrace \alpha_N  |u_{\epsilon}|^{ \frac{N}{N-1} } \right\rbrace dx \leq  \int_{\mathbb{B}} h_0 (x) dx  + \frac{\omega_N}{N} \exp \left\lbrace \alpha_N A_0 + \sum_{j=1}^{N-1} \frac{1}{j} \right\rbrace  .
     \end{flalign*}
     Also we have 
     \begin{flalign*}
     	\underset{\epsilon \rightarrow 0}{\lim \sup} \int_{ \mathbb{B}} h_{\epsilon} (x) \left\lbrace \exp \{ \alpha_N \left( 1 + \epsilon \right) \zeta^{ \frac{N}{N-1} }_{ \epsilon } (x)  \} - 1 \right\rbrace dx  \leq \frac{\omega_N}{N} \exp \left\lbrace \alpha_N A_0 + \sum_{j =1}^{N-1} \frac{1}{j} \right\rbrace .
     \end{flalign*}
\end{proof}

\begin{lemma} \label{lem: 2.8}
	There holds 
	\begin{flalign*}
		\underset{u \in W_0^{1,N} (\mathbb{B}), ||u||_{ W_0^{1,N} ( \mathbb{B} ) } \leq 1 }{\sup} \int_{ \mathbb{B}} h_{\epsilon} (x) \exp \left\lbrace \alpha_N |u|^{ \frac{N}{N-1} } \right\rbrace dx \leq  \int_{\mathbb{B}} h_0 (x) dx +  \frac{\omega_N}{N}  \exp \left\lbrace 3 \alpha_N A_0 + \sum_{j=1}^{N-1} \frac{1}{j} \right\rbrace.
	\end{flalign*}
\end{lemma}

\begin{proof}
	Let $\mathbf{n}$ be the unit outward normal to $\partial \mathbb{B}_{\delta}$ as $0< \delta <1$. We see that
	\begin{flalign*}
		\int_{ \mathbb{B} \setminus \mathbb{B}_{\delta} } G_0 \cdot \operatorname{div} \left( | \nabla G_0 |^{N-2} \cdot \nabla G_0 \right) dx = 0.
	\end{flalign*}
    By the divergence theorem, we have that
    \begin{flalign*}
    	 \int_{ \mathbb{B} \setminus \mathbb{B}_{\delta} } |\nabla G_0|^N dx & =  \int_{\partial  \left( \mathbb{B} \setminus \mathbb{B}_{\delta} \right) } G_0 | \nabla G_0|^{N-2} \frac{ \partial G_0 }{ \partial n} ds - \int_{ \mathbb{B} \setminus \mathbb{B}_{\delta} } G_0 \nabla | \nabla G_0|^{N-2} \cdot \nabla G_0 dx \\
    	 & = G_0 ( \delta ) - \int_{ \mathbb{B} \setminus \mathbb{B}_{\delta} } G_0 \nabla | \nabla G_0|^{N-2} \cdot \nabla G_0 dx.
    \end{flalign*}
    From the definition of Green function, we obtain that
    \begin{flalign*}
    	G_0 (\delta) = - \frac{N}{ \alpha_N } \ln \delta + A_0 + o_{\delta} (1),
    \end{flalign*}
    since $G_0$ is a weak solution of $- \operatorname{div} \left( | \nabla G_0|^{N-2} \cdot \nabla G_0 \right) = \delta_0$. 
    For the second part, we have
    \begin{flalign*}
    	& \int_{ \mathbb{B} \setminus \mathbb{B}_{\delta} } G_0 \nabla | \nabla G_0|^{N-2} \cdot \nabla G_0 + G_0 | \nabla G_0|^{N-2} \Delta G_0 dx \\
    	& = \int_{ \mathbb{B} \setminus \mathbb{B}_{\delta} } G_0 \operatorname{div} \left( |\nabla G_0|^{N-2} \cdot \nabla G_0 \right) dx = 0.
    \end{flalign*}
    Hence, we obtain that 
    \begin{flalign*}
    	\int_{ \mathbb{B} \setminus \mathbb{B}_{\delta}} | \nabla G_0|^N dx = - \frac{N}{\alpha_N} \ln \delta + A_0 + o_{\delta} (1) + o_{\epsilon} (1).
    \end{flalign*}
    With the aid of \eqref{eq: 2.9}, we have that
    \begin{flalign*}
    	\int_{ \mathbb{B} \setminus \mathbb{B}_{\delta}} | \nabla u_{\epsilon}|^N dx = c_{\epsilon}^{- \frac{N}{N-1} } \int_{ \mathbb{B} \setminus \mathbb{B}_{\delta}} \left( | \nabla G_0|^N + o_{\epsilon} (1) \right) dx\\
    	= c_{\epsilon}^{- \frac{N}{N-1} }  \left(- \frac{N}{\alpha_N} \ln \delta + A_0 + o_{\delta} (1) + o_{\epsilon} (1) \right).
    \end{flalign*}
     Denote $\kappa_{\epsilon} = \int_{  \mathbb{B}_{\delta}} | \nabla u_{\epsilon}|^N dx $, it follows that
     \begin{flalign*}
     	\kappa_{\epsilon} & = 1 - \int_{ \mathbb{B} \setminus \mathbb{B}_{\delta}} | \nabla u_{\epsilon}|^N dx \\
     	& = 1- c_{\epsilon}^{- \frac{N}{N-1} }  \left(- \frac{N}{\alpha_N} \ln \delta + A_0 + o_{\delta} (1) + o_{\epsilon} (1) \right).
     \end{flalign*}
     We set $s_{\epsilon} = \underset{\partial \mathbb{B}_{\delta}}{\sup}  u_{\epsilon} = u_{\epsilon}  ( \delta ) $ and $\tilde{u}_{\epsilon} = \left( u_{\epsilon} - s_{\epsilon} \right)^+ $. Obviously, $s_{\epsilon} = c_{\epsilon}^{ - \frac{1}{N-1} } \left( \underset{\partial \mathbb{B}_{\delta} }{\sup } G_0 + o_{\epsilon} (1) \right)$, $\tilde{u}_{\epsilon} \in W_0^{1,N} \left( \mathbb{B}_{\delta} \right)$ and $\int_{\mathbb{B}_{\delta}} | \nabla  \tilde{u}_{\epsilon}|^N dx \leq  \int_{\mathbb{B}_{\delta}} | \nabla u_{\epsilon} |^Ndx $. Then we have that
     \begin{flalign*}
     	\alpha_N\left(1 + \epsilon \right)|u_{\epsilon}|^{ \frac{N}{N-1}  }& \leq \alpha_N \left( 1+ \epsilon \right) \left( s_{\epsilon}  + \tilde{u}_{\epsilon} \right)^{ \frac{N}{N-1} }.
     \end{flalign*} 
     For sufficiently small $ \epsilon$, we have  $\mathbb{B}_{R r_{\epsilon}^{ \frac{1}{1+ \epsilon} }   } \subset \mathbb{B}_{\delta}$, and hence
     \begin{flalign}
     	\int_{\mathbb{B}_{R r_{\epsilon}^{ \frac{1}{1+ \epsilon} }   }} 
     	|x|^{N \epsilon } \exp \{ \alpha_N \left(1 + \epsilon \right) \frac{ |\tilde{u}_{\epsilon}|^{ \frac{N}{N-1} } }{\kappa_{\epsilon} }  \} dx & \leq \int_{\mathbb{B}_{R r_{\epsilon}^{ \frac{1}{1+ \epsilon} }   }} 
     	|x|^{N \epsilon } \exp \{ \alpha_N \left(1 + \epsilon \right) \frac{ |\tilde{u}_{\epsilon}|^{ \frac{N}{N-1} } }{\kappa_{\epsilon} } - 1 \} dx + o_{\epsilon} (1) \notag  \\
     	& \leq \int_{\mathbb{B}_{\delta   }} 
     	|x|^{N \epsilon } \exp \{ \alpha_N \left(1 + \epsilon \right) \frac{ |\tilde{u}_{\epsilon}|^{ \frac{N}{N-1} } }{\kappa_{\epsilon} } - 1 \} dx + o_{\epsilon} (1).\label{eq: 2.10}
     \end{flalign}
     Also we have that
     \begin{flalign}
     	\alpha_N\left(1 + \epsilon \right)|u_{\epsilon}|^{ \frac{N}{N-1}  }& \leq \alpha_N \left( 1+ \epsilon \right) \left( s_{\epsilon}  + \tilde{u}_{\epsilon} \right)^{ \frac{N}{N-1} } \notag \\
     	& \leq \alpha_N (1+ \epsilon) \left\lbrace  ( c_{\epsilon}^{- \frac{2}{N-1}} G^2 (\delta) + 2 c_{\epsilon}^{- \frac{1}{N-1}} G (\delta) \tilde{u}_{\epsilon} +  \tilde{u}_{\epsilon}^2) \right\rbrace ^{ \frac{N}{2 \left(N-1 \right)} } + o_{\epsilon} (1)\notag \\
     	& \leq \alpha_N (1+ \epsilon) \left\lbrace ( 2 G (\delta) + \tilde{u}^2_{\epsilon} ) \right\rbrace ^{ \frac{N}{2 \left(N-1 \right)} } + o_{\epsilon} (1) \notag \\
     	& \leq \alpha_N (1+ \epsilon) (  - 2 \frac{N}{ \alpha_N } \ln \delta + 2 A_0  ) + \alpha_N (1+ \epsilon) \tilde{u}^{ \frac{N}{N-1} }_{\epsilon} + o_{\epsilon}(1).\label{eq: 2.11}
     \end{flalign} 
     From equations \eqref{eq: 2.10} and \eqref{eq: 2.11} together with  the lemma \ref{lem: 2.8}, we obtain
     \begin{flalign}
     	\int_{\mathbb{B}_{R r_{\epsilon}^{ \frac{1}{1+ \epsilon} }   }} 
     	h_{\epsilon} (x) \exp \{ \alpha_N \left(1 + \epsilon \right) |u_{\epsilon}|^{ \frac{N}{N-1} }   \} dx & \leq \int_{\mathbb{B}_{R r_{\epsilon}^{ \frac{1}{1+ \epsilon} }   }} 
     	h_{\epsilon} (x) \exp \{ \alpha_N \left( 1+ \epsilon \right) \left( s_{\epsilon}  + \tilde{u}_{\epsilon} \right)^{ \frac{N}{N-1} }  \} dx \notag  \notag \\
     	& \leq \delta^{ -2N(1+ \epsilon) } \exp \left\lbrace  \alpha_N (1+ \epsilon) ( 2 A_0 + o_{\epsilon}(1) ) \right\rbrace  \notag  \\
     	& \left( \int_{\mathbb{B}_{\delta   }} 
     	|x|^{N \epsilon } \exp \{ \alpha_N \left(1 + \epsilon \right) \frac{ |\tilde{u}_{\epsilon}|^{ \frac{N}{N-1} } }{\kappa_{\epsilon} } - 1 \} dx + o_{\epsilon} (1) \right) \notag \\
     	& \leq \frac{\omega_N}{N} \exp \left\lbrace 2 \alpha_N A_0 +  \sum_{j =1}^{N-1} \frac{1}{j} \right\rbrace \delta^{ -2N(1+ \epsilon) } \cdot  \notag \\
     	& \exp \left\lbrace  \alpha_N (1+ \epsilon) \left( 2 A_0 + o_{\epsilon}(1) \right) \right\rbrace  + o_{\epsilon}(\delta), \label{eq: 2.12}
     \end{flalign}
     with $o_{\epsilon} (\delta) \rightarrow 0$ as $\epsilon \rightarrow 0$ for any fixed $\delta > 0$. Letting $\epsilon \rightarrow 0$ firstly and $R \rightarrow + \infty$ laterly in \eqref{eq: 2.12}, we have
     \begin{flalign*}
     	& \underset{R \rightarrow + \infty}{\lim} \underset{\epsilon \rightarrow 0}{\lim \sup} \int_{\mathbb{B}_{R r_{\epsilon}^{ \frac{1}{1+ \epsilon} }   }} 
     	|x|^{N \epsilon } \exp \{ \alpha_N \left(1 + \epsilon \right) |u_{\epsilon}|^{ \frac{N}{N-1} }   \} dx \\
     	& \leq \frac{\omega_N}{N} \exp \left\lbrace \alpha_N A_0 + \sum_{i=j}^{N-1} \frac{1}{j} \right\rbrace \exp \left\lbrace 2 \alpha_N A_0 \right\rbrace.
     \end{flalign*}
     Then we conclude that
     \begin{flalign*}
     	\underset{\epsilon \rightarrow 0}{\lim \sup} \int_{\mathbb{B}} h_{\epsilon} (x) \exp \left\lbrace \alpha_N (1+ \epsilon) |u_{\epsilon}|^{ \frac{N}{N-1} } \right\rbrace dx \leq \int_{\mathbb{B}} h_0 (x) dx +  \frac{\omega_N}{N}  \exp \left\lbrace 3 \alpha_N A_0 + \sum_{j=1}^{N-1} \frac{1}{j} \right\rbrace.
     \end{flalign*}
     The proof is completed.
\end{proof}

\section{Exclusion of Blow-up Phenomenon}

In this section, we construct a sequence of functions $\phi_{\epsilon} (x) \in W_0^{1,N} \left( \mathbb{B} \right) $ to show that
\begin{flalign*}
	 \int_{\mathbb{B}} h_{\epsilon} (x) \exp \left\lbrace \alpha_N (1+ \epsilon) |\phi_{\epsilon}|^{ \frac{N}{N-1} } \right\rbrace dx > | \mathbb{B}| +  \frac{\omega_N}{N}  \exp \left\lbrace 3 \alpha_N A_0 + \sum_{j=1}^{N-1} \frac{1}{j} \right\rbrace.
\end{flalign*}
From \cite{15}, we set 
\begin{equation*}
	\phi_{\epsilon} (x)= 
	\begin{cases}
		& c+ c^{\frac{1}{ \frac{1}{N-1} }} \left( - \frac{ N-1  }{\alpha_N \left(1 + \epsilon \right) } \ln \left( 1 + c_N |\frac{x}{\epsilon}|^{ \frac{N\left(1 + \epsilon \right)}{N-1} } \right) + b \right), ~ |x| \leq R \epsilon;\\
		& \frac{G}{c^{ \frac{1}{N-1} } }, ~ R \epsilon \leq |x|\leq 1 .
	\end{cases}
\end{equation*}
In order to ensure that $\phi_{k}(x) \in W_0^{1,N} \left( \mathbb{B} \right)$, we require that
\begin{flalign*}
	c+ c^{\frac{1}{ \frac{1}{N-1} }} \left( - \frac{ N-1  }{\alpha_N \left(1 + \epsilon \right) } \ln \left( 1 + c_N |R|^{ \frac{N\left(1 + \epsilon \right)}{N-1} } \right) + b \right) = \frac{G \left( R \epsilon \right)}{c^{ \frac{1}{N-1} } },
\end{flalign*}
which imples that
\begin{flalign*}
	c^{ 1+ \frac{1}{N-1} }= G \left( R \epsilon \right) + \frac{N-1}{\alpha_N \left(1 + \epsilon \right) } \ln \left( 1 + c_N | R|^{ \frac{N \left(1 + \epsilon \right) }{N-1} } \right) - b.
\end{flalign*}

From $\int_{\mathbb{B}} | \nabla \phi_{\epsilon}(x)|^N dx =1$, we have 
\begin{flalign*}
	\int_{|x| \leq R \epsilon } | \nabla \phi_{\epsilon} (x)|^N dx 	& =  c^{ - \frac{N}{N-1} } N^{N-1} \left(N-1 \right) \left\lbrace \ln \left( 1 + c_N R^{ \frac{N}{N-1} } \right) -  \sum_{j=1}^{N-1} \frac{1}{j} + O\left( R^{ - \frac{N \left( 1+ \epsilon \right)}{N-1} } \right)  \right\rbrace,
\end{flalign*}
and
\begin{flalign*}
	\int_{ R \epsilon \leq |x| \leq 1 } | \nabla \phi_{\epsilon}|^N dx = c^{- \frac{N}{N-1} } \left( - \frac{N}{\alpha_N} \ln |R \epsilon| + A_0 + O\left( (R \epsilon)^{2 \gamma} \right) \right),
\end{flalign*}
which together imply
\begin{flalign}
	c^{ \frac{N}{N-1} } & = (N-1) N^{N-1} \left( \ln c_N + \frac{N}{N-1} \ln R - \sum_j \frac{1}{j} + O \left( R^{ \frac{N \left( 1+ \epsilon \right)}{N-1} }\right) \right) \notag \\
	& + G(R \epsilon). \label{eq: 3.1}
\end{flalign}
On the other hand, we obtain
\begin{flalign}
	b= - c^{ \frac{N}{N-1} } + G( R \epsilon) + \frac{N-1}{\alpha_N \left( 1+\epsilon \right)} \ln \left( 1+ c_N R^{ \frac{N \left( 1+ \epsilon \right)}{N-1} } \right). \label{eq: 3.2}
\end{flalign}
In $\mathbb{B}_{ R \epsilon}$, we have
\begin{flalign} \label{eq: 3.3}
	\alpha_N | \phi_{\epsilon}|^{ \frac{N}{N-1} } & = \alpha_N c^{ \frac{N}{N-1} } \left\lbrace  1 + \frac{1}{ c^{\frac{N \left( 1+ \epsilon \right)}{N-1} }  } \left( - \frac{N-1}{ \alpha_N } \ln \left( 1 + c_N |\frac{x}{\epsilon}|^{ \frac{N}{N-1} } \right) + b \right) \right\rbrace^{ \frac{N}{N-1} } \notag \\
	& \geq \frac{N}{N-1} \alpha_N c^{ \frac{N}{N-1} } \left\lbrace  1 + \frac{1}{ c^{\frac{N}{N-1} }  } \left( - \frac{N-1}{ \alpha_N } \ln \left( 1 + c_N |\frac{x}{\epsilon}|^{ \frac{N}{N-1} } \right) + b \right) \right\rbrace + O( R^{ -\frac{N}{N-1} } )\notag \\
	& \geq \alpha_N c^{ \frac{N}{N-1} } + \frac{N \alpha_N b}{ N-1 } - N \ln \left( 1 + c_N | \frac{x}{\epsilon}|^{ \frac{N}{N-1} } \right)+ O( R^{ -\frac{N}{N-1} } ).
\end{flalign}
We consider the right-hand side of the equation \eqref{eq: 3.3},
\begin{flalign*}
	\alpha_N c^{ \frac{N}{N-1} } + \frac{N \alpha_N  }{N-1} b \geq NR + \sum_{j=1}^{N-1} \frac{1}{j} + (N-1) \ln c_N + \alpha_N A_0 + O \left( R^{ - \frac{N}{N-1} } \right),
\end{flalign*}
since
\begin{flalign*}
	\alpha_N\left( c^{ \frac{N}{N-1} } + \frac{N}{N-1} b \right) & = \alpha_N \left\lbrace c^{ \frac{N}{N-1} } + \frac{N}{N-1} \left( - c^{ \frac{N}{N-1} } + G( R \epsilon ) + \frac{N-1}{\alpha_N} \ln \left( 1+ c_NR^{ \frac{N}{N-1} }  \right) \right) \right\rbrace \\
	& = - \ N^{N-1} \left\lbrace \ln \left( 1+ c_N R^{ \frac{N}{N-1} }\right)  - \sum_j \frac{1}{j} + O \left( R^{ - \frac{N}{N-1} }\right)  \right\rbrace\\
	& + \alpha_N G \left( R \epsilon \right) + N \ln \left(  1 + c_N R^{ \frac{N}{N-1} }\right)\\
	& \geq - N^{N-1} \ln \left( 1 + c_N R^{ \frac{N}{N-1} } \right) + \sum_j \frac{1}{j} + O \left( R^{- \frac{N}{N-1}  } \right) - \alpha_N \frac{N}{\alpha_N} \ln | R \epsilon | \\
	& + \alpha_N A_0 .
\end{flalign*}
Moreover, Reader should find that $O(R^{ -\frac{N}{N-1} })$ and $O( R^{ \frac{N \left( 1+ \epsilon \right)}{N-1} } )$ have the same meaning under the asymptotic structure. We consider an approprite $R$ and $\epsilon$ such that
\begin{flalign*}
	R^{ - \frac{N^N}{N-1} } |\epsilon|^{ -N } \rightarrow + \infty
\end{flalign*}
as $R \rightarrow + \infty$, $\epsilon \rightarrow 0$, and $R \epsilon \rightarrow 0$. Then
\begin{flalign*}
	\exp \left\lbrace - N^{N-1} \ln \left( 1+ c_N R^{ \frac{N}{N-1} } \right) - N \ln | R \epsilon| \right\rbrace & = \left( 1 + c_N R^{ \frac{N}{N-1} } \right)^{ - N^{N-1} } \cdot |R \epsilon|^{-N} \rightarrow + \infty.
\end{flalign*}
Moreover, we have
\begin{flalign}
	\int_{|x| < R \epsilon } \exp \left\lbrace - N \ln \left(1 + c_N |\frac{x}{\epsilon} |^{ \frac{N}{N-1} } \right) \right\rbrace dx & = \epsilon^N \int_{ |y| < R } \frac{1}{ \left( 1+ c_N |y|^{ \frac{N}{N-1} } \right)^{ N } } dy \notag \\
	& = \epsilon^N \int_0^{R^{ \frac{N}{N-1} } } \frac{N-1}{N} \frac{ t^{N-2} }{ \left( 1+ c_N t \right)^N } dt \notag \\
	& = \epsilon^N \left( 1 + O \left( R^{ - \frac{N}{N-1} } \right) \right). \label{eq: 3.4}
\end{flalign}
Combining equations \eqref{eq: 3.1}, \eqref{eq: 3.2}, \eqref{eq: 3.3} and \eqref{eq: 3.4}  yields 
\begin{flalign*}
	\int_{ |x| < R \epsilon} \exp \left\lbrace \alpha_N \phi_{\epsilon}^{ \frac{N}{N-1} } \right\rbrace dx & \geq \frac{\omega_N}{N} \exp \left\lbrace 3 \alpha_N A_0 + \sum_{j=1}^{N-1} \frac{1}{j}  \right\rbrace + O \left( R^{ - \frac{N}{N-1} } \right) .
\end{flalign*}
From $e^t \geq 1+t$, we have
\begin{flalign*}
	\int_{ R \epsilon \leq |x| <1 } \exp \left\lbrace \alpha_N \phi_{\epsilon}^{ \frac{N}{N-1} } dx   \right\rbrace > | \mathbb{B}|.
\end{flalign*}
From the two contradictory inequalities, the sequence $u_{\epsilon}$ is uniformly bounded. Elliptic estimates of the equation \eqref{eq: 2.1} then yield that $u_{\epsilon}$ converges to the extremal function in $C^1 \left( \overline{\mathbb{B}} \right)$, and hence Theorem 1.1 is proved.

\textbf{Acknowledgments}

We would like to take this opportunity to thank the editor and the
reviewers again for their constructive comments and useful suggestions, and the time and eﬀorts
they have spent in the review process.

\textbf{Conflict of interest statement}

The authors declare no conﬂict of interest.






\end{document}